\documentclass[11pt,reqno]{amsart}

\usepackage[utf8]{inputenc}
\usepackage[theoremdefs,final]{latexdev}

\usepackage{geometry}                %
\usepackage[numbers]{natbib}

\DeclareMathOperator{\divv}{div}

\newcommand{\pair}[1]{\left\langle #1 \right\rangle}

\providecommand{\norm}[1]{\lVert#1\rVert}
\providecommand{\abs}[1]{\lvert#1\rvert}

\newcommand{\ud}{\mathrm{d}}

\newcommand{\RR}{{\mathbb R}}

\newcommand{\vol}{\mu}

\newcommand{\Diff}{\mathrm{Diff}}
\newcommand{\Xcal}{\mathfrak{X}}
\newcommand{\Diffvol}{{\Diff_{\vol}}}
\newcommand{\Xcalvol}{{\Xcal_{\vol}}}

\newcommand{\Dens}{\mathrm{Dens}}

\newcommand*\id{\mathrm{id}}

\def\ph{\varphi}

\let\on=\operatorname

\title[Semi-invariant Riemannian metrics in hydrodynamics]{Semi-invariant Riemannian metrics in hydrodynamics}

\author[Bauer]{Martin Bauer}
\address[M.\ Bauer]{Department of Mathematics, Florida State University}
\email{bauer@math.fsu.edu}

\author[Modin]{Klas Modin}
\address[K.\ Modin]{Department of Mathematical Sciences, Chalmers University of Technology and the University of Gothenburg}
\email{klas.modin@chalmers.se}

\date{\today}                                           %


\begin{document}

\begin{abstract}
	Many models in mathematical physics are given as non-linear partial differential equation of hydrodynamic type; the incompressible Euler, KdV, and Camassa--Holm equations are well-studied examples.
	A beautiful approach to well-posedness is to go from the Eulerian to a Lagrangian description.
	Geometrically it corresponds to a geodesic initial value problem on the infinite-dimensional group of diffeomorphisms with a \emph{right invariant} Riemannian metric.
	By establishing regularity properties of the Riemannian spray one can then obtain local, and sometimes global, existence and uniqueness results.
	There are, however, many hydrodynamic-type equations, notably shallow water models and compressible Euler equations, where the underlying infinite-dimensional Riemannian structure is not fully right invariant, but still \emph{semi-invariant} with respect to the subgroup of volume preserving diffeomorphisms.
	Here we study such metrics. 
	For semi-invariant metrics of Sobolev $H^k$-type we give local and some global well-posedness results for the geodesic initial value problem.
	We also give results in the presence of a potential functional (corresponding to the fluid's internal energy).
	Our study reveals many pitfalls in going from fully right invariant to semi-invariant Sobolev metrics; the regularity requirements, for example, are higher.
	Nevertheless the key results, such as no loss or gain in regularity along geodesics, can be adopted.
\end{abstract}

\maketitle

\tableofcontents

\section{Introduction}

In 1966 \citet{Ar1966} discovered that the Euler equations of an incompressible perfect fluid can be interpreted as a geodesic equation on the space of volume-preserving diffeomorphisms.
Based on this \citet{EbMa1970} initiated a new approach to local (short time) existence and uniqueness of hydrodynamic PDE.
This approach has since then been extended to many other PDE of mathematical physics, such as the KdV \cite{ovsienko1987super}, Camassa--Holm \cite{camassa1993integrable,kouranbaeva1999camassa}, Hunter--Saxton \cite{HuSa1991,lenells2008hunter}, 
Constantin--Lax--Majda \cite{constantin1985simple,escher2012geometry,bauer2016geometric} and Landau--Lifschitz equations~\cite{ArKh1998}.
The same analysis sometimes lends itself to global existence results \cite{BrVi2014,MiMu2013,MP2010,EK2014a,BEK2015}.

The common setting is a group of diffeomorphisms, thought of as an infinite-dimensional manifold, equipped with a right invariant Riemannian metric.
The pressing issue is to obtain this setting rigorously in the category of Banach manifolds (by Sobolev completion of the diffeomorphism group), and then prove that the PDE becomes an ODE on the Banach manifold with a smooth (or at least Lipschitz continuous) infinite-dimensional vector field.
After that, local existence, uniqueness, and smooth dependence on initial conditions follows from standard results for ODE on Banach manifolds (the Picard--Lindel\"off theorem).
Global results are aquired if the Riemannian structure is strong in the Sobolev topology.

The inner workings of the Ebin and Marsden analysis heavily utilize right invariance of the Riemannian structure---it is through this structure that the PDE (with loss of derivatives) can be formulated as an ODE (without loss of derivatives).
What happens if the Riemannian structure is only semi-invariant?

Here we address this question for the group of diffeomorphisms $\Diff(M)$ (where $M$ is a closed manifold) equipped with a Riemannian metric that is right invariant only with respect to the sub-group of volume-preserving diffeomorphisms $\Diffvol(M)$.
We call such metrics \emph{semi-invariant}.
Our study connects to several lines of research.

\textbf{Optimal transport:} In Otto's~\cite{Ot2001} geometric approach to optimal mass transportation the $L^2$-Wasserstein distance between two probability densities is (formally) obtained as a Riemannian boundary value problem.
The underlying Riemannian structure on the space $P^\infty(M)$ of smooth probability densities stems from a semi-invariant (but not fully invariant) $L^2$-type Riemannian metric on $\Diff(M)$.
Indeed, through Moser's~\cite{Mo1965} result $$P^\infty(M)\simeq \Diff(M)/\Diffvol(M)$$ the $L^2$-type Riemannian structure on $\Diff(M)$ induces, due to the $\Diffvol(M)$-invariance, a Riemannian structure on $P^\infty(M)$.
This structure has low regularity: only Sobolev~$H^{-1}$.
The search for higher order Otto metrics is advocated in the optimal transport community, see \citet[Ch.~15]{Vi2009} and \citet{Sc2018}.
Furthermore, in applications higher regularity than $H^{-1}$ is often desired, for example in imaging to sustain sharp corners.
Higher order Sobolev type Riemannian metrics on $P^\infty(M)$ induced by fully right invariant metrics on $\Diff(M)$ were studied in \cite{BaJoMo2017}.
However, the full invariance imposes restrictions that, for example, excludes the Otto's metric. %
This motivates the study of semi-invariant Riemannian metrics on $\Diff(M)$ of higher regularity than $L^2$.

\textbf{Information geometry:} In the field of information geometry (cf.~\citet{AmNa2000}) the principal Riemannian structure on $P^\infty(M)$ is the Fisher--Rao metric
which induces the (spherical) Hellinger distance. 
The Fisher--Rao metric is canonical in the sense that it is the only Riemannian metric on the space of smooth densities that is invariant under the action of $\Diff(M)$, c.f.\ \cite{cencov2000statistical,AJLV2015,BaBrMi2015}.
It has a similar geometric interpretation as the semi-invariant $L^2$ metric on $\Diff(M)$, only its regularity is higher: $H^1$ instead of $L^2$ \cite{KhLeMiPr2013,Mo2014}.
Thus, our study also contributes towards new Riemannian structures in information geometry.
Furthermore, we can treat the Wasserstein--Otto metric and the Fisher--Rao metric in the same geometric transport framework, allowing mixed order models such as proposed in \cite{chizat2018interpolating,piccoli2014generalized,piccoli2016properties}.%

\textbf{Shallow water equations:} In the field of shallow water equations, the full Euler equations are approximated in the regime where the wave-length is large in comparison to the depth.
The standard shallow water equations for waves evolving on a Riemannian manifold $M$ are
\begin{equation}\label{eq:shallow_water}
\left\{
\begin{aligned}
	& u_t + \nabla_u u + \nabla h = 0 \\
	& h_t + \divv(h u) = 0
\end{aligned}
\right.
\end{equation}
where $u$ is a vector field on $M$ describing the horizontal velocity at the surface and $h(x)$ is the water depth at $x\in M$.
By an analogue to Arnold's interpretation of the incompressible Euler equations, the shallow water equations~\eqref{eq:shallow_water} constitute Newton's equations on $\Diff(M)$, with kinetic energy given by the aforementioned semi-invariant $L^2$ metric and potential energy given by $V(h) = \frac{1}{2}\norm{h}^2_{L^2}$ (see \cite{KhMiMo2018} for details).
A prevailing theme in shallow water research is to modify the equations~\eqref{eq:shallow_water} to obtain more accurate models, for example the Serre--Green--Naghdi (SGN) equations
\begin{equation}\label{eq:sgn}
\left\{
\begin{aligned}
	& u_t + \nabla_u u + \nabla h  = -\frac{1}{3h}\Big(\frac{\partial}{\partial t}+\nabla_u\Big)\nabla^* (h^3 \nabla u) \\
	& h_t + \divv(h u) = 0 ,
\end{aligned}
\right.
\end{equation}
where $\nabla^*$ denotes the $L^2$-adjoint of the covariant derivative.
Following the work of \citet{Io2012}, these equations correspond to Newton's equations on $\Diff(M)$, with the same potential energy as for the classical shallow water equations, but with the modified $H^1$-type kinetic energy
\begin{equation}\label{eq:kinetic_sgn}
	\frac{1}{2}\int_M (h \abs{u}^2 + \frac{h^3}{3}\abs{\nabla u}^2)\vol .
\end{equation}
Since this kinetic energy is quadratic in the vector field $u$, and since $h$ is transported by $u$ as a volume density, it follows that \eqref{eq:kinetic_sgn} corresponds to a semi-invariant $H^1$-type Riemannian metric of $\Diff(M)$.
From a geometric viewpoint, new shallow water models are thus obtained by higher order semi-invariant modifications of the standard $L^2$-type metric on $\Diff(M)$.
This further motivates our study.
One might of course also modify the potential energy as suggested in \cite{ClDuMi2017}.

\textbf{PDE analysis:} From a more mathematical point-of-view, to investigate the degree to which Ebin and Marsden techniques can be extended to the semi-invariant case yields new theoretical insights.
As we shall see, the extension is non-trivial, with some unexpected pitfalls. %
The results of Ebin and Marsden are based on extending Arnold's Riemannian metric on $\Diff_\vol(M)$ to a Sobolev completion $\mathcal D_\vol^s(M)$ of the diffeomorphisms.
If $s>\operatorname{dim}(M)/2 + 1$ then $\mathcal D_\vol^s(M)$ is a Banach manifold.
Remarkably, the associated (infinite-dimensional) Riemannian spray on $\mathcal D_\vol^s(M)$ is then smooth, so local well-posedness follows from standard ODE-theory on Banach manifolds (see e.g.\ \cite{Lan1999}).
The Ebin and Marsden approach has successfully been extended to the Sobolev completion $\mathcal D^s(M)$ of all diffeomorphisms for right invariant Sobolev $H^k$-metric for $k\geq 1$ \cite{MP2010,kolev2017local}.
Furthermore, for strong Riemannian metrics, i.e., where $k=s$, the right invariance of the metric yields global well-posedness \cite{GAYBALMAZ2015}.
In contrast, our study shows that not every smooth semi-invariant $H^1$-metric yields a smooth spray (the SGN metric \eqref{eq:kinetic_sgn} is an example)
and global results are not readily available without additional assumptions.
Even more, for local results the Sobolev completion $\mathcal D^s(M)$ of $\Diff(M)$ requires a higher Sobolev index $s$ than in the fully right invariant case.
Nevertheless, with modifications the key components of the Ebin and Marsden technique \emph{can} be adopted to the semi-invariant case, for example the no-loss-no-gain result (see \autoref{app:noloss} below).

\subsection{Main results}
Let $(M,g)$ be a closed (compact and without boundary) oriented Riemannian manifold of finite dimension~$d$.
Associated with the metric $g$ is the Riemannian volume form $\vol$, and the Levi--Civita covariant derivative $\nabla$ (acting on tensor fields). 
We denote time derivatives by subscript $t$, for example $u_t = \partial u/\partial t$.

The Riemannian metric $g$ can be extended to arbitrary $r$-$q$-tensors via
\begin{equation}
g^q_r = \bigotimes^{r} g  \otimes \bigotimes^{q} g^{-1}\;.
\end{equation}
To simplify the notation we write $g$ also for this extended metric.
For tensor fields $X,Y\in C^{\infty}(M,\mathcal T^q_r(M))$ we use vector calculus notation $X\cdot Y \coloneqq g(X,Y)$ and $\abs{X} \coloneqq \sqrt{g(X,X)}$.
The $L^2$ inner product on tensor fields is given by
\begin{equation}
	\pair{X,Y}_{L^2} = \int_M X\cdot Y \,\vol .
\end{equation}

The space of smooth vector fields on $M$ is denoted $\Xcal(M)$.
Furthermore, the space of smooth probability densities is given by
\begin{equation}
	P^\infty(M) = \{ \rho\in C^\infty(M)\mid \rho >0, \int_M \rho \vol = 1 \}.
\end{equation}

Consider the following family of Lagrangian functionals on the hydrodynamic phase space $\Dens(M)\times \Xcal(M)$
\begin{equation}
	L(u,\rho) = \frac{1}{2}\sum_{i=0}^k \int_M a_i\circ\rho \, \abs{\nabla^i u}^2 \, \vol - V(\rho)
\end{equation}
where $V\colon P^\infty(M)\to \RR$ is a potential functional, and $a_i\colon\RR_{>0}\to \RR_{\geq 0}$ are smooth coefficient functions that specify how the kinetic energy depends on the density variable~$\rho$.
Since the kinetic energy is quadratic on the vector field $u$, the variational derivative of $L$ (with respect to the $L^2$ inner product) is a family of differential operators $A\colon P^\infty(M)\times\Xcal(M)\to\Xcal(M)$ that is linear in the vector field component $u$
\begin{equation}
	\frac{\delta L}{\delta u} \eqqcolon A(\rho)u= \sum_{i=0}^k (\nabla^i)^*\,a_i\circ\rho\, \nabla^i u\,.
\end{equation}
The other variational derivative, with respect to $\rho$, is
\begin{equation}
	\frac{\delta L}{\delta \rho} \eqqcolon B(\rho,u) = \frac{1}{2}\sum_{i=0}^k a_i'\circ\rho \, \abs{\nabla^i u}^2 - \frac{\delta V}{\delta \rho} \, .
\end{equation}
From these derivatives we obtain a corresponding family of hydrodynamic-type PDE in the time-dependent vector fields $u$ and $m$ and density $\rho$
\begin{equation}\label{eq:main_eq}
\left\{
\begin{aligned}
	&m_t + \nabla_u m+(\mathrm{div}\hspace{0,1cm}u)m+(\nabla u)^\top m -\rho \nabla B(\rho,u) = 0  \\
	&\rho_t+\divv(\rho u)  = 0 \\
	& m = A(\rho)u  \\
	& u|_{t=0} = u_0, \quad \rho|_{t=0} = \rho_0 .
\end{aligned}
\right.
\end{equation}
From the point of view of analytical mechanics, the variable $m$ is the momentum associated with the fluid velocity $u$, and $A(\rho)$ is the inertia operator.

Before formulating the main result we list some special cases of the equations~\eqref{eq:main_eq}.

\begin{example}	
	If $V=0$ and the coefficient functions $a_i$ are constants, so that $B(\rho,u)\equiv 0$, we obtain the EPDiff equation~\cite{holm2005momentum}.
	This corresponds to a fully right invariant Riemannian structure on $\Diff(M)$. 
	For $M=S^1$ we obtain as special cases the Camassa--Holm \cite{kouranbaeva1999camassa},
the Constantin--Lax--Majda~\cite{escher2012geometry} and the Hunter--Saxton equation \cite{lenells2008hunter}.

\end{example}
\begin{example} 
	If $V=0$ and $k=0$ with $a_0(r) = r$ we obtain Burgers' equation
	\begin{equation}
		u_t + \nabla_u u = 0 ,
	\end{equation}
	which is the simplest hydrodynamic model (fluid particles are moving along geodesics on $M$ without interacting with each other).
	This corresponds to the semi-invariant $L^2$-metric whose distance is the classical $L^2$-Wasserstein distance (also called `earth-movers distance').
\end{example}
\begin{example}
	If the potential functional is $$V(\rho)= \int_M e(\rho)\rho \,\vol$$ for some internal energy function $e(\rho)$, and $k=0$ with $a_0(r) = r$, we obtain the (barotropic) compressible Euler equations
	\begin{equation}
		\left\{
		\begin{aligned}
			& u_t + \nabla_u u  + \frac{1}{\rho}\nabla e'(\rho)\rho^2 = 0 ,\\
			& \rho_t + \divv(\rho u) = 0 .
		\end{aligned}
		\right.
	\end{equation}
	The choice $e(\rho) = \rho/2$ coincides with the classical shallow water equations \eqref{eq:shallow_water}, where the density $\rho$ then is the water depth (denoted $h$ in \eqref{eq:shallow_water}).
\end{example}
\begin{example} 
	If the potential functional is $$V(\rho) = \frac{1}{2}\int_M \rho^2\,\vol ,$$ and $k=1$ with $a_0(r) = r$ and $a_1(r) = r^3/3$, we obtain the SGN equations~\eqref{eq:sgn}, again with $\rho$ as the depth function $h$.
\end{example}

\begin{maintheorem*}
	Consider the equations \eqref{eq:main_eq} with $\rho_0\in P^\infty(M)$ and $u_0\in\Xcal(M)$, and a potential functional $V$ such that
	\begin{equation}
		\frac{\delta V}{\delta \rho}\colon P^\infty(M)\to C^\infty(M)
	\end{equation}
	is a smooth (nonlinear) differential operator of order $2k-2$ or less.
	\begin{itemize}
		\item If $k=1$ with $a_0(\cdot)>0$ and $a_1(\cdot) = \text{const} > 0$, or
		\item if $k>2$ with $a_0(\cdot)>0$ and $a_k(\cdot)>0$,  
	\end{itemize}
	then there exists a unique solution defined on a maximal time-interval of existence $J\subset \RR$, which is open and contains zero.
	The solution $u=u(t,x)$ and $\rho=\rho(t,x)$ depends smoothly (in the Fréchet topology of smooth functions) on the initial conditions. 

	Furthermore, if $k>d/2+1$ and $V=0$, and if $a_0(\cdot) > C_1$ and $a_k(\cdot) > C_2$ for constants $C_1,C_2>0$, then $J=\RR$, i.e., we have global existence.
\end{maintheorem*}
In the remainder of the paper we prove this and other related results in the more general setting when $A(\rho)$ is an elliptic differential operator fulfilling certain assumptions.

\medskip
\smallskip
\noindent\textbf{Acknowledgements.}
We are grateful to Dimitrios Mitsotakis who pointed us to the geometric interpretation of the Serre--Green--Naghdi equation.
We would also like to thank Sarang Joshi and François-Xavier Vialard for helpful discussions.
The second author was supported by EU Horizon 2020 grant No~691070, by the Swedish Foundation for International Cooperation in Research and Higher Eduction (STINT) grant No~PT2014-5823, and by the Swedish Research Council (VR) grant No~2017-05040.

\section{Semi-invariant Riemannian metrics on diffeomorphisms}\label{Sec:SDiffinvarianmetrics}

\subsection{Background on diffeomorphism groups}\label{sec:diffgroups}

The group $\Diff(M)$ of all smooth diffeomorphisms is an infinite-dimensional Fréchet Lie group, i.e., it is a Fréchet manifold and the group operations (composition and inversion) are smooth maps \cite[\S\!~I.4.6]{Ha1982}.
The corresponding Fréchet Lie algebra is the space $\Xcal(M)$ of smooth vector fields equipped with minus the vector field bracket.

To obtain results on existence of geodesics on $\Diff(M)$, the standard approach is to work in the Banach topology of Sobolev completions, and then use a `no-loss-no-gain' in regularity result by \citet{EbMa1970}.
 Therefore we introduce 
\begin{align}
\mathcal{D}^s(M)=\left\{ \varphi\in H^s(M,M)\mid \varphi \text{ is bijective and } \varphi^{-1}\in H^s(M,M)\right\},\quad s>\frac{d}{2}+1\;,
\end{align}
which is a Hilbert manifold and a topological group.
It is, however, not a Lie group, since left multiplication is not smooth (only continuous).
The corresponding set of Sobolev vector fields is denoted $\Xcal^s(M)$. 
For a detailed treatment we refer to the research monograph by \citet*{IKT2013}.

\subsection{Geodesic equation}
In the following let $G$ be a Riemannian metric on $\Diff(M)$ that is invariant with respect to the right action of the volume preserving diffeomorphism group 
$\operatorname{Diff}_\vol(M)$, but not necessary with respect to arbitrary diffeomorphisms in $\Diff(M)$, i.e.,
\begin{equation}\label{eq:sdiff_inv_metric}
G_{\varphi}(h,k)=G_{\varphi\circ\psi}(h\circ\psi,k\circ\psi) 
\end{equation}
for all $\varphi\in \Diff(M), h,k\in T_{\varphi}\Diff(M)$ and $\psi\in \operatorname{Diff}_\vol(M)$.
We refer to such a metric as \emph{semi-invariant}.

Let $\rho=\det(D\varphi^{-1})$.
Then any Riemannian metric of the form 
\begin{equation}\label{eq:right_inv_metricG}
G_{\varphi}(u\circ\varphi,v\circ\varphi) %
=\int_M u\cdot  A(\rho)v\, 
\vol\qquad \forall u,v \in \Xcal(M)
\end{equation}
is semi-invariant, i.e., satisfies the invariance property~\eqref{eq:sdiff_inv_metric}.
Here, the \emph{inertia operator}
$$A(\rho)\colon \Xcal(M)\to \Xcal(M)$$ is a field of operators that are self-adjoint (with respect to the $L^2$ inner product) and positive.

\begin{assumption}\label{assump:inertia}
	The inertia operator $A(\rho)$ fulfills these conditions:
	\begin{enumerate}
		\item For a fixed integer $k\geq 1$ %
		the map
		\begin{equation}
			(\rho,u)\mapsto A(\rho)u
		\end{equation}
		is a smooth differential operator 
		$P^{\infty}(M)\times \Xcal(M)\to \Xcal(M)$ of order $2k-2$ in its first argument and of order $2k$ in its second argument.
		\item For any $\rho\in P^{\infty}(M)$ the map
		\begin{equation}
			u\mapsto A(\rho)u
		\end{equation}
		is a linear positive elliptic differential operator which is self-adjoint with respect to the $L^2$ inner product.
		\item Let $v\mapsto A'(\rho)^*(u,v)$ be the $L^2$ adjoint of the $\rho$-derivative $\dot\rho\mapsto A'(\rho)(u,\dot\rho)$, i.e.,
		\begin{equation}\label{eq:weakAprime}
		\int_M A'(\rho)^*(u,v) \cdot w\; \vol =\int_M v \cdot \left( A'(\rho)(u,w \right)\vol\;.
		\end{equation}
		Then the mapping
		\begin{equation}
			(\rho,u,v) \mapsto A'(\rho)^*(u,v)
		\end{equation}
		is a smooth differential operator $P^{\infty}(M)\times \Xcal(M)\times \Xcal(M)\to C^\infty(M)$ of order $2k-2$ in its first argument and of order $2k-1$ in its second and third arguments. 
	\end{enumerate} 	
\end{assumption} 
\begin{example}\label{rem:form:A}
An important family of inertia operators is given by
\begin{equation}\label{inertia_operator}
A(\rho)= \sum_{i=0}^k (\nabla^i)^*a_i(\rho) \nabla^i\;,
\end{equation}
where $a_i$
are smooth coefficient functions depending on $\rho$. 
This class of inertia operators stems from semi-invariant Riemannian metrics of the form
\begin{equation}\label{eq:metric:A1}
	G_{\varphi}(X\circ\varphi,Y\circ\varphi)= \sum_{i=0}^k \int_M a_i(\rho) \;g(\nabla^i X,\nabla^i Y) \vol\;.%
\end{equation}
Such metrics are common in shallow water equations and in regularized compressible fluid equations. 
We study metrics of this type in \autoref{geodesic_completeness} and we show in \autoref{lem:conditions:standardexample} below that they satisfy \autoref{assump:inertia} under mild conditions on the 
coefficient functions $a_i$.
\end{example}

We now give the geodesic equation.
\begin{theorem}\label{thm:geodesicequationDiff}
The geodesic equation of an $\Diffvol(M)$-invariant Riemannian metric on the group of smooth diffeomorphisms whose inertia operator fulfills \autoref{assump:inertia} is given by the PDE
\begin{align}
&\big(A(\rho)u\big)_t + \nabla_u A(\rho)u+(\mathrm{div}\hspace{0,1cm}u)A(\rho)u+(\nabla u)^\top A(\rho)u -\frac{\rho}2 \nabla \big(A'(\rho)^*(u,u)\big)=0 \label{eq:epdiff3} \\
&\rho_t+\divv(\rho u)  = 0\label{eq:epdiff2}
\end{align}
where 
$v\mapsto A'(\rho)^*(u,v)$ is the adjoint of the $\rho$-derivative $\dot\rho\mapsto A'(\rho)(u,\dot\rho)$ as in \eqref{eq:weakAprime}.

The geodesic $t\mapsto  \varphi(t,\cdot)$ on $\Diff(M)$ is reconstructed from a solution $(\rho,u)$ by
\begin{equation}
	\ph_t = u \circ\ph .
\end{equation}
As before, the probability function $\rho$ is related to $\varphi$ via $\rho = \det(D\ph^{-1})$.
\end{theorem}
Using the notation $A(\rho)u=m$ one can rewrite \eqref{eq:epdiff3} to obtain:
\begin{equation}
m_t + \nabla_u m+(\mathrm{div}\hspace{0,1cm}u)m+(\nabla u)^\top m -\frac{\rho}2 \nabla \big(A'(\rho)^*(u,u)\big)=0 \label{eq:epdiff4a} 
\end{equation}
From here, one can easily deduce the similarities to the EPDiff equation: the last term in \eqref{eq:epdiff4a} is new. 

\begin{proof}[Proof of \autoref{thm:geodesicequationDiff}]
The energy of a path of diffeomorphisms $ \ph= \ph(t,x)$ is given by
\begin{align}
E(\ph)&=\int_0^1 G_{\ph}(\ph_t,\ph_t)\ud t %
\\
&=\int_0^1 \int_M g\left(A({\rho})( \ph_t\circ\ph^{-1}) ,  \ph_t\circ\ph^{-1}\right)\, \vol \;\ud t
\end{align}
Varying $\ph$ in the direction $h = h(t, x)$ with $h(0,\cdot) = h(1,\cdot) = 0$ we calculate
\begin{equation}
\ud (\ph_t\circ\ph^{-1}).h=h_t\circ\ph^{-1}-\nabla\ph_t\circ\ph^{-1} .(\nabla\ph)^{-1}h\circ\ph^{-1}\;.
\end{equation}
Thus we obtain for the variation of the energy functional:
\begin{align}
\ud E(\ph).h&= 2\int_0^1 \int_M g\left(A({\rho})( \ph_t\circ\ph^{-1}) ,  h_t\circ\ph^{-1}-\nabla\ph_t\circ\ph^{-1}.(\nabla\ph)^{-1}h\circ\ph^{-1}  \right)\, \vol \;\ud t\\
&\qquad+\int_0^1 \int_M g\left(\ud A'({\rho})( \ph_t\circ\ph^{-1}, \ud \rho . h) ,  \ph_t\circ\ph^{-1}\right)\vol \;\ud t
\end{align}
Using the notation $u=\ph_t\circ\ph^{-1}$ we can rewrite this to obtain:
\begin{align}\label{eq:first_term}
\ud E(\ph).h&= 2\int_0^1 \int_M g\left(A(\rho)u ,  h_t\circ\ph^{-1}-\nabla u .h\circ\ph^{-1}  \right)\, \vol \;\ud t\\
&\qquad+\int_0^1 \int_M g\left(A'(\rho)( u,\ud \rho . h) ,  u\right)\vol \;\ud t \label{eq:second_term}
\end{align}
It remains to separate all terms involving the variation $h$. 
In the following we will treat the two terms separately. For the first term \eqref{eq:first_term} we use
\begin{align}
h_t\circ\ph^{-1} &= (h\circ\ph^{-1})_t-\nabla h\circ\varphi^{-1} \partial_t(\ph^{-1})\\
&=(h\circ\ph^{-1})_t+\nabla h\circ\varphi^{-1} (\nabla\ph)^{-1}\circ\ph^{-1}.u\\
&=(h\circ\ph^{-1})_t-\nabla (h\circ\varphi^{-1}).u
\end{align}
Thus we obtain for the first term
\begin{align}
2\int_0^1 \int_M g\left(A(\rho)u ,  (h\circ\ph^{-1})_t-\nabla (h\circ\varphi^{-1}).u-\nabla u .h\circ\ph^{-1}  \right)\, \vol \;\ud t
\end{align}
Using integration  by parts and writing $m=A(\rho)u$ we can rewrite this to obtain
\begin{align}
-2\int_0^1 \int_M g\left(m_t+(\nabla u)^Tm +\nabla_u (m)+(\mathrm{div}\hspace{0,1cm}u)(m) ,  (h\circ\ph^{-1})\right)\, \vol \;\ud t
\end{align}
For the second term \eqref{eq:second_term} we use the variation formula of $\rho$ in direction $h$ (see e.g.\ \cite{BaJoMo2017})
\begin{equation}
\ud\rho.h= -\on{div}(\rho (h\circ\varphi^{-1}))\;.
\end{equation}
Thus, for the second term we have
\begin{align}
\int_0^1 \int_M g\left(A'(\rho)( u,\ud \rho . h) ,  u\right)\vol \;\ud t
&= \int_0^1 \int_M g\left(A'(\rho)( u,-\divv(\rho (h\circ\varphi^{-1}))) ,  u\right)\vol \;\ud t \\
&= \int_0^1 \int_M - \divv(\rho (h\circ\varphi^{-1}))A'(\rho)^*(u,u)\vol \;\ud t \\
&= \int_0^1 \int_M g\Big(h\circ\varphi^{-1}, \rho \nabla A'(\rho)^*(u,u) \Big)\vol \;\ud t .
\end{align}
The conbination of the first and second term now gives equation~\eqref{eq:epdiff3}.
Using that $$\rho_t=-\on{div}(\rho u)$$ we also obtain equation~\eqref{eq:epdiff2}.
\end{proof}

The following theorem on local well-posedness is our first main result.

\begin{theorem}[Local well-posedness]\label{Thm:localwellposedness}
Let $G$ be the $\Diff_{\vol}(M)$-invariant metric \eqref{eq:right_inv_metricG} with inertia operator $A(\rho)$ satisfying 
\autoref{assump:inertia}. Then, given any $(\varphi_0,v_0)\in T\Diff(M)$, there exists a unique non-extendable geodesic
$(\varphi(t),v(t))\in C^{\infty}(J,T\Diff(M))$
on the maximal interval of existence $J$, which is open and contains zero.
\end{theorem}

We postpone the details of the proof to the end of~\autoref{wellposedness:Ds}. 
In essence, the main idea is to extend the metric and the geodesic spray to the Sobolev 
completion $\mathcal D^s(M)$.
This allows us to interpret the geodesic equation as an ODE on a Hilbert manifold and thus to use the theorem of Picard--Lindel\"off to prove local well-posedness on the completion. 
The statement in the smooth category then follows by a no-loss-no-gain result.
By a small modification of the proof one also obtains the corresponding result with a potential.

\begin{corollary}
	Let $G$ be as in \autoref{Thm:localwellposedness} and let $V\colon P^\infty(M)\to \RR$ be a potential functional such that its variational derivative $\delta V/\delta \rho$ is a smooth (non-linear) differential operator of order $2k-2$ or less.
	Then the statement of \autoref{Thm:localwellposedness} is valid also for the flow of the Lagrangian on $T\Diff$ given by
	\[
		L(\varphi,\dot\varphi) = \frac{1}{2} G_{\varphi}(\dot\varphi,\dot\varphi) - V(\det(D\varphi^{-1})).
	\]
\end{corollary}

In the following section we discuss geodesic completeness, i.e., the question of global in time existence of solutions to the geodesic initial value problem.

\subsection{Geodesic completeness}\label{geodesic_completeness}
For geodesic completeness we will focus on the class of operators introduced in \autoref{rem:form:A}, i.e.,
\begin{equation}\label{inertia_operator_repeatdef}
A(\rho)= \sum_{i=0}^k (\nabla^i)^*a_i(\rho) \nabla^i\;,
\end{equation}
where $a_i\in C^{\infty}(\mathbb R_{>0},\mathbb R_{\geq 0})$
are smooth coefficient functions depending on $\rho$. 
For this class we are able to prove geodesic completeness---the second of our main results.
First we show that the class satisfies \autoref{assump:inertia}:
\begin{lemma}\label{lem:conditions:standardexample}
The operator \eqref{inertia_operator_repeatdef} satisfies \autoref{assump:inertia} if one of the conditions
\begin{enumerate}
\item $k=1$, $a_0(\rho)>0$ and $a_1(\rho)=\operatorname{const}>0$; 
\item $k\geq 2$, $a_0(\rho)>0$ and $a_k(\rho)>0$; 
\end{enumerate}
is satisfied.  %
\end{lemma}
\begin{proof}
It is straight-forward to see that $A(\rho)$ satisfies item (1) and (2) of \autoref{assump:inertia}. 
Note, that we use here that $a_1$ does not depend on $\rho$ for the case of an operator with $k=1$.
To see that it satisfies item (3) we calculate an explicit expression of 
the $L^2$ adjoint: 
\begin{align}
		\int_M A'(\rho)^*(v,w) \cdot w\; \vol &=\int_M v \cdot \left( A'(\rho)(w,u) \right)\vol\;.\\	
	&=\int_M v \cdot \left(  \sum_{i=0}^k (\nabla^i)^*a_i'(\rho)\operatorname{div}(u) \nabla^i w \right)\vol	
	\\&= 	\sum_{i=0}^k \int_M a_i'(\rho) \nabla^ i v \cdot \nabla^i w \; \operatorname{div}(u)  \vol\\
	&= 	\sum_{i=0}^k \int_M   u \cdot \nabla \left(a_i'(\rho) \nabla^ i v \cdot \nabla^i w  \right) \vol\,.
\end{align}
Thus we have:
\begin{equation}
A'(\rho)^*(v,w)=  \sum_{i=0}^k \nabla \left(a_i'(\rho) \nabla^ i v \cdot \nabla^i w \right).
\end{equation}
Counting derivatives we obtain that this is a smooth differential operator $P^{\infty}(M)\times \Xcal(M)\times \Xcal(M)\to C^\infty(M)$ of order one in its first argument  and of order $k+1$ in its second and third argument. Thus all assumptions are satisfied.
\end{proof}
Using this lemma we can directly apply the local  well-posedness result of the previous section to obtain the result:
\begin{corollary}
Let $A$ be an inertia operator of the form  \eqref{inertia_operator_repeatdef} that satisfies either condition (1) or (2) of \autoref{lem:conditions:standardexample}. Then the geodesic equation of the corresponding $\Diffvol(M)$-invariant Riemannian metric on the group of smooth diffeomorphisms 
\begin{align}
&\big(A(\rho)u\big)_t + \nabla_u A(\rho)u+(\mathrm{div}\hspace{0,1cm}u)A(\rho)u+(\nabla u)^\top A(\rho)u -\frac{\rho}2 \nabla \left(\sum_{i=0}^k \nabla \left(a_i'(\rho) \abs{\nabla^ i u}^2 \right)\right)=0 \label{eq:epdiff4} \\
&\rho_t+\divv(\rho u)  = 0\label{eq:epdiff5}
\end{align}
is locally well-posed in the sense of \autoref{Thm:localwellposedness}. 
\end{corollary}
The particular form of this initial operator allows in addition to obtain a global well-posedness result and to characterize the metric completion:
\begin{theorem}[Global well-posedness]\label{Thm:globalwellposedness}
	Let $G$ be a $\Diff_{\vol}(M)$-invariant metric~\eqref{eq:right_inv_metricG} of order $k> \frac{d}2+1$ with inertia operator $A(\rho)$ 
	of the form \eqref{inertia_operator_repeatdef} with $a_1(\cdot)>C_1$ and $a_k(\cdot)>C_2$ for some constants $C_1,C_2>0$.
	We have:
	\begin{enumerate}
	\item The space $\left(\Diff(M),G\right)$ is geodesically complete, i.e., for any initial condition $(\varphi_0,v_0)\in T\Diff(M)$, the unique geodesic 
	$(\varphi(t),v(t))\in C^{\infty}(\mathbb R,T\Diff(M))$ with $(\varphi(0),v(0)) = (\varphi_0,v_0)$ exist for all time $t$.
	\item The metric completion of the space $\left(\Diff(M),\operatorname{dist}^G\right)$ is the space of all Sobolev diffeomorphisms $\mathcal{D}^k(M)$ of regularity $k$, as defined in \autoref{sec:diffgroups}. Here $\operatorname{dist}^G$ denote the induced geodesic distance function. 
	\end{enumerate}
\end{theorem}
The proof is postponed to the end of~\autoref{wellposedness:Ds}.

\subsection{Sobolev completion}\label{wellposedness:Ds}

We shall now extend the metric and the geodesic spray to the Sobolev completion $\mathcal D^s(M)$ of $\Diff(M)$ and establish smoothness results.
The theorem of Picard and Lindel\"off then gives local well-posedness of the geodesic initial value problem.
From smoothness of the extended metric and additional conditions on the inertia operator we also obtain global well-posedness.

\begin{theorem}\label{thm:smooth_extension}
Let $G$ be a $\Diff_{\vol}(M)$-invariant metric of the form~\eqref{eq:right_inv_metricG}.%
\begin{enumerate}
\item If the inertia operator $A(\rho)$ satisfies 
	\autoref{assump:inertia} then $G$ extends to a smooth $\mathcal D^s_{\vol}(M)$-invariant metric on 
 $\mathcal D^s(M)$ for every $s>\frac{d}2+2k$.
	\item If the inertia operator $A(\rho)$ is of the form \eqref{inertia_operator_repeatdef} with $k>\frac{d}{2}+1$ then $G$ extends to a  a smooth $\mathcal D^s_{\vol}(M)$-invariant metric on 
 $\mathcal D^s(M)$ for every $s>\frac{d}2+1$ and $s-k\geq 0$. 
 If in addition  $a_0(\cdot)>C_1$ and $a_k(\cdot)>C_2$ for some constants $C_1,C_2>0$ then it extends to a strong Riemannian metric on the Sobolev completion
 $\mathcal D^k(M)$.
	 \end{enumerate}
\end{theorem}

\begin{remark}
For right invariant metrics on $\Diff(M)$ it is easy to show that the theorem holds if the inertia operator $A$ is a differential operator, c.f. \cite{EbMa1970}. For more general inertia operators the situation is  more complicated:  for Pseudodifferential Operators that are represented as a Fourier multipliers the statement has been proven on the diffeomorphism group of the circle and of $\mathbb R^d$, see \cite{EK2014,BEK2015}. For  arbitrary compact manifolds no results beyond the class of differential operators are known. 
\end{remark}
\begin{proof}
The metric in terms of $h,k\in T_{\varphi}\mathcal D^s(M)$ is given by
\begin{equation}
G_{\varphi}(h,k)= \int_M g(A(\rho)(h\circ\varphi^{-1}),k\circ\varphi^{-1})\vol\;.
\end{equation}
Since inversion $\varphi\mapsto \varphi^{-1}$ is only continous, but not $C^1$, the smoothness of the metric is not clear a priori.

To show the smoothness of the metric we adopt the strategy devised for $\Diff(M)$-invariant metrics on $\mathcal D^s(M)$: by a change of coordinates we write the metric as
\begin{equation}
G_{\varphi}(h,k)= \int_M g((A(\rho)(h\circ\varphi^{-1}))\circ\varphi,k)\varphi^*\vol=
\int_M g((R_{\varphi} \circ A(\rho)\circ R_{\varphi^{-1}})(h),k)\varphi^*\vol\;.
\end{equation}
Since the mapping $\varphi\mapsto \varphi^*\vol = \det(D\varphi)\vol$ is smooth for $s>\frac{d}{2}+1$ the smoothness of the metric reduces to the smoothness of the bundle operator 
\begin{equation}
	T\mathcal D^s(M) \to T^{s-2k}\mathcal D^s(M);\qquad (\varphi,h) \mapsto A_{\varphi}h:=(R_{\varphi} \circ A(\rho)\circ R_{\varphi^{-1}})(h)\;,
\end{equation}
where $T^{s-2k}\mathcal D^s(M)$ denotes the bundle above $\mathcal D^s(M)$ whose fiber vectors belong to $H^{s-2k}(M,TM)$ (see \citet{EbMa1970} for details on this infinite dimensional vector bundle). 
Now the smoothness statement follows directly from \autoref{lem:smoothrho} using \autoref{assump:inertia}.
Note that we need only a slightly weaker condition to get smoothness of the metric, namely  $A$ can be of one order higher in $\rho$ as compared to the assumptions.

Using the particular form of the inertia operator in item (2) of the theorem we are able to significantly weaken the assumptions on the Sobolev order $s$. 
Therefore we  decompose $A$ as
\begin{equation}
A(\rho)= \sum_{i=0}^{k} (\nabla^i)^* \circ M_{a_i(\rho)}\circ \nabla^i,
\end{equation}
with the 
$M_{a_i(\rho)}$ operator being multiplication by $a_i(\rho)$.  
We then have
\begin{align}
&(A)_{\varphi}=R_{\varphi} \circ A(\rho)\circ R_{\varphi^{-1}} = 
R_{\varphi} \circ \left(\sum_{i=0}^{k}  (\nabla^i)^* \circ M_{a_i(\rho)}\circ \nabla^i\right) \circ R_{\varphi^{-1}}\\&\qquad=
\sum_{i=0}^{k}  \left(\left(R_{\varphi} \circ (\nabla^i)^* \circ R_{\varphi^{-1}}\right) \circ \left(R_{\varphi} \circ   M_{a_i(\rho)} \circ R_{\varphi^{-1}} \right)\circ \left(R_{\varphi} \circ \nabla^i \circ R_{\varphi^{-1}}\right)\right)
\end{align}
We can thus show the desired  smoothness result by analyzing the components. The operators $\nabla^i$ and its adjoint are constant coefficient differential operators, thus we can use the results  for fully right invariant metric to  obtain that the mappings
\begin{align}
T\mathcal D^{s}(M)\to T^{s-i}\mathcal D^s(M);\qquad (\varphi,h) \mapsto (R_{\varphi} \circ \nabla^i \circ R_{\varphi^{-1}})(h)\\
T^{s-i}\mathcal D^{s}(M)\to T^{s-2i}\mathcal D^s(M);\qquad (\varphi,h) \mapsto (R_{\varphi} \circ (\nabla^i)^*\circ R_{\varphi^{-1}})(h)\\
\end{align}
are smooth for $1\leq i\leq k$, see e.g.\ \cite{EbMa1970}. 
It remains to study the conjugation of the multiplication operator. Here we use that
\begin{equation}
R_{\varphi} \circ   M_{\rho} \circ R_{\varphi^{-1}}= M_{\rho\circ\varphi}\;.
\end{equation}
Thus, using the proof of \autoref{lem:smoothrho}, the smoothness of the mapping 
\begin{align}
\mathcal D^{s}(M)\to L(H^{q}(M,TM),H^{q}(M,TM));\qquad \varphi \mapsto M_{\rho\circ\varphi}
\end{align}
follows for $0\leq q\leq s-1$. Here we used the Sobolev embedding theorem.
It remains to show that $G$ is a strong metric on the Sobolev completion $\mathcal D^{k}(M)$. Since we assumed that $a_k(\rho)>C_2>0$ and $a_0(\rho)>C_1>0$ it follows that the metric is uniformly stronger than the $\Diff(M)$-right invariant Sobolev metric $\bar G$ with inertia operator $1+\Delta^k$. 
Thus the metric $G$ is a strong Riemannian metric on $\mathcal D^{k}(M)$ as this statement holds for the metric $\bar G$, c.f. \cite{BrVi2014}.%
\end{proof}

We are now in position to prove local well-posedness of the geodesic equation on the Sobolev completions.
\begin{theorem}\label{wellposednessSob}
Let $G$ be the $\Diffvol(M)$-invariant metric with inertia operator $A(\rho)$ fulfilling \autoref{assump:inertia}.
Let $s>\frac{d}2+2k$ . 
Then, given any $(\varphi_0,v_0)\in T\mathcal D^{s}(M)$, there exists a unique non-extendable geodesic
$(\varphi(t),v(t))\in C^{\infty}(J^s,T\mathcal D^{s}(M))$
on the maximal interval of existence $J^s$, which is open and contains zero.
\end{theorem} 
\begin{proof}
	We prove that the geodesic equation is an ODE on the Banach manifold $T\mathcal D^s(M)$ for a smooth vector field $F$ on $T\mathcal D^s(M)$ of the form
	\begin{equation}
	F(\varphi,h) = (\varphi, h,h, S_{\varphi}(h)) \in T(T\mathcal D^s(M)),
	\end{equation}
	where $S_{\varphi} = R_{\varphi}\circ S \circ R_{\varphi^{-1}}$ with
	\begin{equation}\label{eq:spray}
	\begin{aligned}
	S(u)&=\nabla_u u +  A(\rho)^{-1}\bigg(-\nabla_u m-(\mathrm{div}\hspace{0,1cm}u)(m)-(\nabla u)^\top (m) \\&\qquad\qquad\qquad\qquad\qquad+
	\frac{\rho}2 \nabla \big(A'(\rho)^*(u,u)\big) 
	+ A'(\rho)(u,\divv(\rho u))
	\bigg)\;.
	\end{aligned}
	\end{equation}
	Note the new term $\nabla_u u$ as compared to the geodesic equation \eqref{eq:epdiff3} expressed in $u$. 
	Incorporating this term we get
	\begin{equation}
	\begin{aligned}
	S(u)&=  A(\rho)^{-1}\bigg([A(\rho),\nabla_u] u-(\mathrm{div}\hspace{0,1cm}u)(A(\rho)u)-(\nabla u)^\top (A(\rho)u) \\&\qquad\qquad\qquad\qquad\qquad+\frac{\rho}2 \nabla \big(A'(\rho)^*(u,u)\big) 
	+ A'(\rho)(u,\divv (\rho u)) \bigg)\;.
	\end{aligned}
	\end{equation}
	To simplify the presentation, we introduce the notation
	\begin{align}
	&Q^1(\rho, u)= [A(\rho),\nabla_u] u,\qquad  Q^2(\rho, u)= (\mathrm{div}\hspace{0,1cm}u) A(\rho)u,\qquad Q^3(\rho, u)= (\nabla u)^\top (A(\rho)u),\\
	&Q^4(\rho, u)=  \frac{\rho}2 \nabla \Big(A'(\rho)^*(u,u)\Big) 
	+ A'(\rho)(u,\divv(\rho u)).
	\end{align}
	Then 
	\begin{equation}
	S_{\varphi}(h)= (A)_{\varphi}^{-1}\left(Q^1_{\varphi}(h)+Q^2_{\varphi}(h)+Q^3_{\varphi}(h)+Q^4_{\varphi}(h)\right),
	\end{equation}
	where 
	$$Q^i_{\varphi}(h)=Q^i(\rho, h\circ\varphi^{-1})\circ\varphi\;, \qquad \rho=\det(D\varphi^{-1})\;.$$
	It is immediate that $Q^2,\ldots,Q^4$ are smooth differential operators $P^\infty(M)\times \Xcal(M) \to \Xcal(M)$ of order $2k-1$ in their first argument and order $2k$ in their second argument.
	Thus, for $i=2,3,4$ it follows from \autoref{lem:smoothrho} (see \autoref{app:smooth_lemma} below) that $(\varphi,h)\mapsto Q^i_\varphi(h)$ are smooth as maps $T\mathcal D^s(M)\to T^{s-2k}\mathcal D^s(M)$.

	That $Q^1$ is a smooth differential operator of order $2k$ is more intricate.
	For this, we use that the differentiating part of the differential operator $u\mapsto \nabla_v u$ acts diagonally on $u$ and is tensorial (takes no derivatives) in $v$.
	Indeed, it is of the form $\nabla_v u = K_1(v)u + \sum_{i}(K_2(v) u^i) \frac{\partial}{\partial x^i}$, where $(u,v)\mapsto K_1(v)u$ is a bilinear tensorial map and $K_2(v)$ is tensorial in $v$ and a scalar differential operator of order 1 in $u^i$ (see, e.g., \cite[p.~\!1317]{Mo2014} for details).
	Consequently, the commutator is of the form 
	\begin{equation}
		[A(\rho),\nabla_u]u = [A(\rho), K_1(u)]u + \sum_{ij} ([\alpha^{i}_j(\rho), K_2(u)]u^j)\frac{\partial}{\partial x^{i}}
	\end{equation}
	where $(\rho,f)\mapsto \alpha^i_j(\rho)f$ are differential operators of order~$2k-2$ in $\rho$ and $2k$ in $f$.
	The first commutator term is of order $2k-2$ in $\rho$ and $2k$ in $u$ since $K_1$ is tensorial.
	The other commutators are also of order $2k$ in $u$, since the commutator between two scalar differential operators of order $2k$ and $1$ gives a differential operator of order~$2k$ (the $2k+1$ order derivatives cancel).
	However, we loose one derivative in $\rho$, so the remaining commutators are of order $2k-1$ in $\rho$.
	Using \autoref{lem:smoothrho} this proves that $(\varphi,h)\mapsto Q^1_\varphi(h)$ is smooth as a mapping $T\mathcal D^s(M)\to T^{s-2k}\mathcal D^s(M)$.
	
	It remains to show that $A_\varphi^{-1}$ is a smooth mapping $T^{s-2k}\mathcal D^s(M)\to T\mathcal D^s(M)$.
	Using that $A(\rho)$ is a $2k$-safe operator in the terminology of \cite{mue2017} allows us to use elliptic regularity theory for differential operators with non-smooth 		  coefficients see \cite{mue2017}.  Thus we obtain that 
	 the operator $A(\rho)\colon\mathfrak X^s(M)\to \mathfrak X^{s-2k}(M)$ is invertible with inverse 
	$A(\rho)\colon\mathfrak X^{s-2k}(M)\to \mathfrak X^{s}(M)$
	by~\autoref{assump:inertia}. 
	From here it follows that $A_\varphi\colon \mathfrak X^s(M)\to \mathfrak X^{s-2k}(M)$
	is invertible as well since
	\begin{equation}
	A_{\varphi}^{-1} = (R_{\varphi}\circ A(\rho)\circ R_{\varphi^{-1}})^{-1}=R_{\varphi}\circ A(\rho)^{-1}\circ R_{\varphi^{-1}}\;.
	\end{equation} 
	Since  the map 
	\begin{align}
T\mathcal D^{s}(M)\to T^{s-2k}\mathcal D^s(M);\qquad (\varphi,h) \mapsto A_{\varphi}(h)
\end{align}
is smooth, c.f. the proof of \autoref{thm:smooth_extension}, the smoothness of the map 
\begin{align}
T^{s-2k}\mathcal D^{s}(M)\to T\mathcal D^s(M);\qquad (\varphi,h) \mapsto A^{-1}_{\varphi}(h)
\end{align}
follows, which concludes the proof.	
\end{proof}

The proof of \autoref{wellposednessSob} can easily be modified to handle also a potential functional as in the Main Theorem in the introduction.

\begin{corollary}
	Let $G$ be as in \autoref{wellposednessSob} and let $V\colon P^{s-1}(M)\to \RR$ be a potential functional such that its variational derivative $\delta V/\delta \rho$ is a smooth (non-linear) differential operator of order $2k-2$ or less.
	Then the statement of \autoref{wellposednessSob} is valid also for the flow of the Lagrangian on $T\mathcal{D}^{s}$ given by
	\[
		L(\varphi,\dot\varphi) = \frac{1}{2} G_{\varphi}(\dot\varphi,\dot\varphi) - V(\det(D\varphi^{-1})).
	\]
\end{corollary}

\begin{proof}
	What changes in the equations when introducing the potential is the operator $Q^4$ in the proof of \autoref{wellposednessSob}.
	It becomes
	\[
		Q^4(\rho,u) = \rho \nabla \left(\frac{1}{2}A'(\rho)^*(u,u) - \frac{\delta V}{\delta\rho}(\rho)\right) 
	+ A'(\rho)(u,\divv(\rho u)).
	\]
	If $\delta V/\delta \rho$ is a smooth differential operator of order $2k-2$ then $Q^4$ remains of order $2k-1$ in $\rho$, as needed in the proof of \autoref{wellposednessSob}.
\end{proof}

With stronger assumptions on the inertia operator $A$ we are able to prove metric and geodesic completeness.
\begin{theorem}\label{thm:completenessSob}
Let $G$ be a $\Diff_{\vol}(M)$-invariant metric of order $k> \frac{d}2+1$ with inertia operator $A(\rho)$ 
of the form \eqref{inertia_operator_repeatdef}
with $a_1(\cdot)>C_1$ and $a_k(\cdot)>C_2$ for some constants $C_1,C_2>0$.
We have:
\begin{enumerate}
\item The space $\left(\mathcal D^{k}(M),\operatorname{dist}^G\right)$ is a complete metric space, where $\operatorname{dist}^G$ denotes the induced geodesic distance function. 
\item For any $s\geq k$ and any $(\varphi_0,v_0)\in T \mathcal D^{s}(M)$, the unique geodesic
$(\varphi(t),v(t))\in C^{\infty}(\mathbb R,T \mathcal D^{s}(M))$
exist for all time $t\in \mathbb R$, i.e., the space $\left(\mathcal D^{s}(M), G\right)$ is geodesically complete.
\end{enumerate}
\end{theorem}
\begin{proof}
By \autoref{thm:smooth_extension} the metric $G$ extends to a smooth Riemannian metric on the Sobolev completion $\mathcal D^{s}(M)$, for $s\geq k$ and to a smooth and  strong 
Riemannian metric for $s=k$. Since we assumed that $a_k(\rho)>0$ and $a_0(\rho)>0$ it follows that the metric is uniformly stronger than the $\Diff(M)$ right invariant Sobolev metric $\bar G$ with inertia operator $1+\Delta^k$.  
This implies that the space $\left( \mathcal D^{s}(M),d^{G}\right)$ is a complete metric space.
Here we used that this statement holds for the metric $\bar G$, c.f.~\cite{BrVi2014}. 
By \cite{Lan1999} metric completeness of strong Riemannian metrics implies geodesic completeness.\footnote{This is the only statement of the Hopf--Rinow theorem that holds in infinite dimensions.} 
Thus the geodesic initial value problem  on the space $\mathcal D^{k}(M)$ is globally (in time) well-posed. The result for $\mathcal D^{s}(M)$ with $s\geq k$ follows by the no-loss-no-gain result \autoref{lem:nolossnogain} (as in \citet{EbMa1970}).
\end{proof}
The proof of our main results (local and global well-posedness in the smooth category) now follows from a no-loss-no-gain result in \autoref{app:noloss}.
\begin{proof}[Proof of \autoref{Thm:localwellposedness} and \autoref{Thm:globalwellposedness}]
The proof of this result is now an immediate consequence of \autoref{thm:completenessSob}, \autoref{wellposednessSob} and \autoref{lem:nolossnogain} in \autoref{app:noloss} below.
\end{proof}

\subsection{Outlook for semi-invariant metrics} \label{sec:outlook}

The results in this section point toward a systematic study of semi-invariant Riemannian metrics on $\Diff(M)$.
Indeed, inspired by developments for right invariant metrics, there are several follow-ups. 
For example:
\begin{itemize}
	\item A study of sectional curvature and Fredholm properties of the Riemannian exponential, such as carried out for fully right invariant metrics by \citet{MP2010}.
	This has direct implications on the stability of perturbations.

	\item Stronger results on geodesic completeness. 
	It is likely that our result can be extended to a much larger class of semi-invariant metrics.
	In particular, a setting of \autoref{thm:smooth_extension} that replaces the condition $a_k(\cdot)>C_2$ with the more natural one $a_k(\cdot) > 0$.
	Our result now is based on domination by a right invariant metric for which geodesic completeness holds.

	\item A study of gradient flows on $P^\infty(M)$, as Otto~\cite{Ot2001} did for the $L^2$ metric, for general semi-invariant metrics.

	\item An investigation of vanishing geodesic distance. 
	There is a general result that fully right invariant metrics always have positive geodesic distance if the order is high enough. 
	The right invariant $L^2$ metric is known to have vanishing geodesic distance (that is, any two points can be joined by a geodesic that can be made arbitrarily short).
	Are there corresponding results for semi-invariant metrics?

	\item In light of the geometric interpretation of shallow water equations as Newton-type systems on $\Diff(M)$, one could consider higher-order metrics as a way to obtain more accurate shallow water models.

\end{itemize}

\section{Application: Riemannian metrics on probability densities}\label{sec:dens}

We mentioned already in the introduction that an important application of semi-invariant metric on $\Diff(M)$ is that they induce new Riemannian structures on the space $P^\infty(M)$ of probability densities.
Indeed, the resulting geometry on $P^\infty(M)$ can be interpreted as a generalized optimal transport model. 
In this section we give formulas for the induced metric, and we give existence results based on the theorems of \autoref{Sec:SDiffinvarianmetrics}.
The difficulty from the point-of-view of analysis is that one cannot directly work with Sobolev completions of $P^\infty(M)$; one has to work on $\Diff^s(M)$, then let $s\to\infty$, and then project to $P^\infty(M)$.
Thus, the no-loss-no-gain theorem in \autoref{app:noloss} is essential here.

The results we present here, for semi-invariant metric, are straightforward adaptations of the results for right invariant metrics presented in \cite{BaJoMo2017}.

\subsection{Background on probability densities}
An extended version of this subsection, containing all proofs, can be found in \cite{BaJoMo2017}. 
Here we will only present the results needed in the remainder.

The space $P^\infty(M)$ of smooth probability densities is naturally equipped with an infinite-dimensional Fréchet topology, making it a \emph{Fréchet manifold} \cite[\S\!~III.4.5]{Ha1982}.
Its tangent bundle is thereby also a Fréchet manifold, and the tangent spaces are given by
\begin{equation}
	T_\rho P^\infty(M) = \{ f\in C^\infty(M)\mid \int_M f \vol = 0 \}.
\end{equation}
Notice that $T_\rho P^\infty(M)$ is independent of $\rho$; this is because $P^\infty(M)$ is an open subset of an affine space.

Analogous to the situation for the group of diffeomorphisms, the space of Sobolev probability densities
\begin{align}
 P^{s}(M)&=\{\rho \in H^{s}(M) \mid \int_M \rho\, \vol =1,\quad \rho>0 \},\quad s>\frac{d}{2}\;
\end{align}
is a Banach manifold.
From the point-of-view of the Sobolev embedding theorem, we need $s>\frac{d}{2}$ to ensure that the condition $\rho>0$ is well-defined point-wise.

The group of diffeomorphisms 
$\Diff(M)$ acts on $P^\infty(M)$ from the left by pushforward of densities
\begin{align}
\Diff(M)\times P^\infty(M)&\mapsto P^\infty(M)\\
(\varphi,\rho)&\rightarrow \det(D\varphi^{-1})\rho\circ\varphi^{-1}  .
\end{align}
By a result of \citet{Mo1965} this actions is transitive.
Thus, the action on the unit dentity $\rho\equiv 1$ yields a projection $\Diff(M)\to P^\infty(M)$.
We shall need the following result of Hamilton.
\begin{theorem}\label{thm:principal_bundle_diff_dens}
The set of volume preserving diffeomorphisms
\begin{equation}
	\Diffvol(M) = \{\varphi\in\Diff(M)\mid \det(D\varphi^{-1}) \equiv 1 \}
\end{equation}
is a closed Fréchet Lie subgroup.
Furthermore, the projection
\begin{align}
\pi: \on{Diff}(M)&\mapsto P^\infty(M)\\
\varphi&\mapsto \det(D\varphi^{-1})\;,
\end{align}
is a smooth principal $\Diffvol(M)$-bundles over $P^\infty(M)$ with respect to the left action of $\Diffvol(M)$ on $\Diff(M)$.
Hence, the set of left cosets $$\Diff(M)/\Diffvol(M)$$ is identified with $P^\infty(M)$ by $\pi$.
\end{theorem}

For the projections $\pi$  we can calculate the corresponding vertical bundles, defined by the kernel of the tangent mapping.
\begin{lemma}\label{lem:vertical_bundle}
The vertical bundle of the projection $\pi$ is given by
\begin{align}
	\operatorname{Ver}_{\varphi}
	&=\left\{\dot\varphi \in T_\varphi\Diff(M)\mid  \operatorname{div}(\rho u) = 0, \; u\coloneqq \dot\varphi\circ\varphi^{-1},\; \rho \coloneqq \det(D\varphi^{-1}) \right\}.
\end{align}
\end{lemma} 
As the proof of this lemma contains an important calculation for the remainder we repeat it here.
\begin{proof}
To calculate the differential of the projection mapping let $\phi(t,\cdot)$ be a path of diffeomorphisms  with 
\begin{align}
\phi(0,\cdot)&=\varphi\\
\partial_t\big|_{t=0}\phi &= h := u \circ\varphi \text{ for some } u\in \Xcal(M).                                                            
\end{align}
We then use 
\begin{align}
0= \partial_t\big|_{t=0}\left(\phi(t)^*\phi(t)_*\vol\right)=\phi(t)^*\left(\mathcal L_u \varphi_*\vol \right)+
\phi(t)^*\partial_t\big|_{t=0}\left( \varphi_*\vol \right)
\end{align}
to obtain
\begin{align}\label{derivative_pi_l}
T_{\varphi}\pi(u\circ\varphi)=\partial_t\big|_{t=0}\left( \phi_*\vol \right)=-\mathcal L_u \varphi_*\vol
= -\on{div}(\rho u)\vol
\end{align}
where $\rho=\det(D\varphi^{-1})$. %
\end{proof}
The projection  $\pi$ can also be extended to the Sobolev category. 
It turns out, however, that this extension is continous but not smooth:
\begin{lemma}\label{lem:vertical}
Let $s>\frac{d}{2}+1$ and let $\pi$  be the  projection as defined in \autoref{thm:principal_bundle_diff_dens}.
Then $\pi$ 
extends to a surjective mapping
\begin{equation}\label{projection_left}
\begin{aligned}
\pi^s: \mathcal{D}^s(M)&\mapsto P^{s-1}(M)\\
\varphi&\mapsto \det(D\varphi^{-1})\;.
\end{aligned}
\end{equation}
This mapping is $C^0$ but not $C^1$.
\end{lemma}

\subsection{Induced metric}

We shall now calculate the induced metric on  $P^\infty(M)$ for a semi-invariant metric on $\Diff(M)$ corresponding to an operator $A(\rho)$ as in \eqref{eq:right_inv_metricG}.

We first address the question of existence of the horizontal bundle.
\begin{lemma}\label{lem:hor_exists}
Let $G$ be an $\operatorname{Diff}_\vol(M)$-invariant metric on $\Diff(M)$ of the form~\eqref{eq:right_inv_metricG}.
Then the horizontal bundle with respect to the projection $\pi$ exists in the Fréchet topology as a complement of the vertical bundle~$\on{Ver}^l$. 
It is given by
\begin{align}
 \operatorname{Hor}_{\varphi}
 &=\left\{\left(A(\rho)^{-1} (\rho \nabla p)\right)\circ\varphi\mid p\in C^{\infty}(M)\right\},
\end{align}
where $\rho=\det(D\varphi^{-1})$. 
Thus, every vector $X\in T_\varphi\Diff(M)$ has a unique decomposition
$X=X^{\operatorname{Ver}}+X^{\operatorname{Hor}}$ with $X^{\operatorname{Ver}}\in\operatorname{Ver}_{\operatorname{\varphi}}(\pi)$ and $X^{\operatorname{Ver}}\in \operatorname{Hor}_{\varphi}(\pi)$.
\end{lemma}

\begin{proof}
Let $h=u\circ\varphi \in T_{\varphi}\Diff(M)$. 
Then $h\in \operatorname{Hor}_{\varphi}(\pi)$ if and only if
\begin{equation}
G_{\varphi}(h,k)=0, \qquad \forall k\in  \operatorname{Ver}_{\varphi}(\pi)\;.
\end{equation}
Let $k=v\circ\varphi$ and $\rho = \det(D\varphi^{-1})$. 
Using the characterization of $\operatorname{Ver}_{\varphi}(\pi)$ in \autoref{lem:vertical_bundle} this yields
\begin{equation}\label{eq:hor_char}
G_{\varphi}(h,k)=\int_M g(A(\rho)u,v) \vol=0, \qquad \forall v\in \Xcal(M) \text{ with } \operatorname{div}(\rho v)=0 
\end{equation}
Consider now the vector field $w=\frac1{\rho}A({\rho})u$.
The Hodge decomposition for $w$ yields
\begin{equation}
w=\nabla p + \tilde w\;.
\end{equation}
with unique components $p\in  C^{\infty}(M)/\RR$ and $\tilde w\in \Xcalvol(M) = \{u\in\Xcal(M)\mid\divv u =0 \}$.
Thus, we can decompose $u$ as
\begin{align}\label{eq:u_Ainv}
u= A(\rho)^{-1}(\rho\nabla p +\rho \tilde w)
\end{align}
with both $A(\rho)^{-1}(\rho\nabla p)$ and $A(\rho)^{-1}(\rho \tilde w)$ in $\Xcal(M)$. 
Plugging \eqref{eq:u_Ainv} into \eqref{eq:hor_char} then yields
\begin{equation}
G_{\varphi}(h,k)=\int_M g(\rho\nabla p +\rho \tilde w,v) \vol = \int_M g(\rho\nabla p,v) \vol +\int_M g(\rho \tilde w,v) \vol 
\end{equation}
Using integration by parts, the first term vanishes
\begin{equation}
 \int_M g(\rho\nabla p,v) \vol =\int_M g(\nabla p,\rho v) \vol=-\int_M  p\operatorname{div}(\rho v) \vol=0.
\end{equation}
Thus, $k=u\circ\varphi$ is horizontal if $u$ is of the form $A(\rho)^{-1}(\rho\nabla f)$.
It remains to show that if $\tilde w\neq 0$, then $u\circ\varphi$ is not horizontal.
For this, we note that $v=\frac{1}{\rho}\tilde w$ satisfies $\operatorname{div}(\rho v)=0$ and
\begin{equation}
	\int_M g(\rho\tilde w,v)\vol = \norm{\tilde w}_{L^2}^2 .
\end{equation}
This concludes the characterization of the horizontal bundle.
\end{proof}

A consequence of \autoref{lem:hor_exists} is that the Riemannian metric $G$ induces a Riemannian metric on $P^\infty(M)$.
To see what the induced metric is, we need to calculate the horizontal lift of a tangent vector $\dot \rho\in T_\rho P^\infty(M)$.
To this end we introduce a field of pseudo differential operators over $P^\infty(M)$ given by

\begin{equation}\label{Lrho}
\bar A({\rho})^{-1}\colon\begin{cases} 
      C^{\infty}(M)/\mathbb R&\longrightarrow C^{\infty}_0(M)  \\
      p&\longmapsto -\on{div}(\rho A^{-1}({\rho})(\rho\nabla p))\;.\\
   \end{cases}   
\end{equation}
Geometrically, one should think of the field $\bar A({\rho})^{-1}$ as the inverse of a Legendre transform (we shall see later that it actually is), identifying (the smooth part of) the cotangent bundle $T^* P^\infty(M)\simeq P^\infty(M)\times C^\infty(M)/\RR$ with the tangent bundle $T P^\infty(M)\simeq P^\infty(M)\times C^\infty_0(M)$.

\begin{lemma}\label{lem:Lrho_inverse}
Let $A(\rho)$ be a field of positive, elliptic, differential operators of order $2k$, self-adjoint with respect to 
the $L^2$ inner product. 
For any $\rho\in P^\infty(M)$ the pseudo differential operator operator $\bar A({\rho})^{-1}$ of order $-2k+2$ defined in \eqref{Lrho} is an isomorphism.
\end{lemma}

\begin{proof}
Using integration by parts, $\bar A({\rho})^{-1}$ is self adjoint since $A(\rho)^{-1}$ is.
For any $q$ in $\mathbb N \cup \infty$ we can extend $\bar A({\rho})^{-1}$ to a bounded linear operator $H^q(M)/\mathbb R\to H^{q+2k-2}_0(M)$.  
To prove that $\bar A({\rho})^{-1}$ is an elliptic operator, we decompose it in its components
\begin{align}
\bar A({\rho})^{-1}=  -\on{div}\circ M_{\rho}\circ A({\rho})^{-1} \circ M_{\rho} \circ\nabla\;,
\end{align}
where $M_{\rho}$ is the multiplication operator with $\rho$. 
$M_{\rho}$ is elliptic since $\rho(x)>0$ for all $x\in M$.
Thus, $\bar A({\rho})^{-1}$ is weakly elliptic as it is a composition of weakly elliptic operators; here one uses the fact that the principal symbol is multiplicative, see \cite[Sect.~4]{LM1989}.
As a next step, we want to determine the kernel of $\bar A({\rho})^{-1}$. 
\begin{align}
\int_M  \bar A({\rho})^{-1}(p) p\, \vol = -\int_M \on{div}(\rho A(\rho)^{-1}(\rho\nabla p)) p \vol
=\int_M g(A(\rho)^{-1}(\rho\nabla p)),\rho \nabla p) \vol>0 
\end{align}
for all $p\neq [0] \in H^q(M)/\RR$. 
Here we use that 
\begin{align}
\int_M g(A(\rho)^{-1}u,u) \vol>0
\end{align}
for all $u\in \Xcal(M)\backslash \{ 0\}$.
Thus $\bar A({\rho})^{-1}$ is injective, as it is strictly positive on $H^q(M)/\mathbb R$. 
Since it is Fredholm with index zero it is also surjective.
The isomorphism result is valid for smooth functions due to elliptic regularity, see \cite[Sect.~5]{LM1989}.
\end{proof}

We now obtain an isomorphism between $\on{Hor}_\varphi$ and $T_{\pi(\varphi)}P^\infty(M)$.

\begin{lemma}\label{hor:lift}
Let $G$ be a $\operatorname{Diff}_\vol$-invariant metric on $\Diff(M)$ of the form~\eqref{eq:right_inv_metricG}.
Then 
\begin{equation}
	T_\varphi\pi|_{\on{Hor}_\varphi}\colon \on{Hor}_\varphi \to T_{\pi(\varphi)} P^\infty(M)
\end{equation}
is an isomorphism.
The inverse is given by
\begin{equation}
T_{\pi(\varphi)}P^\infty(M) \ni \dot \rho \mapsto A(\rho)^{-1}(\rho \nabla p)\circ\varphi\in\on{Hor}_\varphi,
\end{equation}
where %
\begin{equation}
p=\bar A({\rho})\dot \rho\;.
\end{equation}
\end{lemma}
\begin{proof}
The horizontal lift of a tangent vector $\dot \rho \in T_{\rho}P^\infty(M) $ is the unique horizontal vector field $u$ such that
\begin{align}
T_{\varphi}\pi (u\circ\varphi)= \dot \rho
\end{align}
where $\varphi$ is some diffeomorphism with $\pi(\varphi)=\rho$.
Using the characterization of the horizontal bundle and the formula for $T\pi$ this yields the equation
\begin{align}
\dot \rho = -\on{div}^{\rho\vol}(u) \det(D\varphi^{-1})= -\on{div}^{\rho\vol}( A(\rho)^{-1}(\rho\nabla p))\det(D\varphi^{-1}) = -\on{div}(\rho A(\rho)^{-1}\rho\nabla p)
\end{align}
The above lifting equation can be rewritten as
\begin{align}
\dot \rho = -\on{div}(\rho A(\rho)^{-1}(\rho\nabla p))=\bar A({\rho})^{-1}(p).
\end{align}
Applying $\bar A({\rho})$ to the above equation yields the desired result.
\end{proof}

Using \autoref{hor:lift} we obtain the formula for the induced metric on $P^\infty(M)$. 
\begin{proposition}\label{pro:descending_metric_formula}
Let $G$ be a  $\operatorname{Diff}_\vol$-invariant metric on $\Diff(M)$ of the form~\eqref{eq:right_inv_metricG},
with inertia operator $A({\rho})$ as in \autoref{lem:Lrho_inverse} of order $2k$.
Then the induced  metric on $P^\infty(M)$ is given by
 \begin{align}\label{met:dens}
\bar G_{\rho}(\dot \rho,\dot \rho)&= %
\int_M  (\bar A(\rho) \dot \rho)\, \dot\rho\,\vol . %
\end{align}
The pseudo-differential operator $\bar A({\rho})$ is of order $2k-2$, so $\bar G$ is of order $k-1$.
\end{proposition}

\begin{proof}
From \autoref{hor:lift} for the horizontal lift of a tangent vector we get
\begin{align}
\bar G_{\rho}(\dot \rho ,\dot \rho )&= G_\id( \rho \nabla \bar A_{\rho}(\dot \rho), \rho \nabla \bar A_{\rho}(\dot \rho)) = \int_M g\left(   \rho \nabla \bar A_{\rho}(\dot \rho), A^{-1}\left( \rho \nabla \bar A_{\rho}(\dot \rho)  \right)\right )\vol.
\end{align}
Using  integration by parts and that $A$ is self-adjoint we obtain
\begin{align}
G_{\rho}(\dot \rho, \dot \rho)&= \int_M g\left(  \on{div}  \rho A^{-1} \rho \nabla \bar A_{\rho}(\dot \rho),    \bar A_{\rho}(\dot \rho) \right )\vol
\end{align}
Since $\bar A_{\rho}^{-1}=\on{div}  \rho A^{-1} \rho \nabla$ this equals
 \begin{align}
G_{\rho}(\dot \rho, \dot \rho)&= \int_M g\left(  \bar A_{\rho}^{-1}\bar A_{\rho}(\dot \rho),   \bar A_{\rho}(\dot \rho) \right )\;\vol\\
&=\int_M g\left(  \dot \rho,   \bar A_{\rho}(\dot \rho) \right )\;\vol\;.
\end{align}
The order of the pseudodifferential operator $\bar A_\rho$ follows by counting derivatives.
\end{proof}
The following lemma connects (local and global) well-posedness of the geodesic initial value problem on
$\Diff(M)$ to well-posedness on $P^\infty(M)$ equipped with the induced quotient metric.
\begin{lemma}\label{thm:wellposednessdiffwellposednessprob2}
Let $G$ be a  $\operatorname{Diff}_\vol$-invariant metric on $\Diff(M)$ of the form~\eqref{eq:right_inv_metricG}. 
Assume that  given any $(\varphi_0,v_0)\in T\Diff(M)$, there exists a unique non-extendable geodesic
$(\varphi(t),v(t))\in C^{\infty}(J,T\Diff(M))$
defined on the maximal interval of existence $J$, which is open and contains zero.

Let $\bar G$ be the induced metric on $P^\infty(M)$.
Then, given any $(\rho_0,\dot p_0)\in T P^\infty(M)$, there  
exists a unique non-extendable geodesic
$(\rho(t),p(t))\in C^{\infty}(J,T P^\infty(M))$
defined on the same maximal interval of existence $J$.
\end{lemma}
\begin{proof}
To prove this result, we need, for any choice of initial data $(\rho_0,p_0)$, to construct a solution $(\rho,p) \in C^{\infty}(J,T P^\infty(M))$ to the geodesic initial value problem  with $\rho(0)=\rho_0$ and $p(0) = p_0$. To this end, let $\varphi_0$ be an arbitrary diffeomorphism such that $\varphi_{0*}\vol = \rho_0\vol$.
Let $\dot\rho_0 = \bar A({\rho_0})^{-1}p_0$.
Using \autoref{hor:lift} we can lift $(\rho_0,\dot\rho_0)$ to a unique horizontal vector $u_0\circ \varphi_0 \in T_{\varphi_0}\Diff(M)$.
Thus, using the assumptions, we obtain a unique solution $(\varphi,v)$ with $\varphi(0) = \varphi_0$ and $v(0) = \dot\varphi_0\circ\varphi_0^{-1}=v_0$ on a non-empty maximal existence interval $J$ containing $0$.
Since $\dot\varphi(0) \in \on{Hor}_{\varphi(0)}$ it follows that $\dot\varphi(t) \in \on{Hor}_{\varphi(t)}$ for every $t\in J$ and thus it projects to a geodesic on $P^\infty(M)$, since the two metrics are related by a Riemannian submersion. From here the result follows.
\end{proof}
As an immediate consequence is the following theorem on local and global well-posedness on the space of densities.

\begin{theorem}\label{thm:wellposednessdiffwellposednessprob}
Let $G$ be a  $\operatorname{Diff}_\vol$-invariant metric on $\Diff(M)$ of the form~\eqref{eq:right_inv_metricG} and let 
$\bar G$ be the induced metric on $P^\infty(M)$. We have:
\begin{enumerate}
\item If the inertia operator $A(\rho)$ satisfies \autoref{assump:inertia} then the geodesic initial value problem on 
$ P^\infty(M)$ is locally well-posed, i.e, given any $(\rho_0,\dot p_0)\in T P^\infty(M)$ there  
exists a unique non-extendable geodesic
$(\rho(t),p(t))\in C^{\infty}(J,T P^\infty(M))$
defined on the  maximal interval of existence $J$, which is open and contains zero.

\item Let $A(\rho)$ be of the form \eqref{inertia_operator_repeatdef} with $a_1(x)>C_1$ and $a_k(x)>C_2$ for some constants $C_1,C_2>0$ and with $k> \frac{d}2+1$. 
Then the space $\left( P^\infty(M),\bar G\right)$ is geodesically complete, i.e., for any initial condition 
$(\rho_0,\dot p_0)\in T P^\infty(M)$ there  
exists a unique geodesic
$(\rho(t),p(t))\in C^{\infty}(J,T P^\infty(M))$
with interval of existence $J=\mathbb R$.
\end{enumerate}
\end{theorem}

\appendix
\section{No-loss-no-gain for semi-invariant flows}\label{app:noloss}
In the following we will show that the class of $\operatorname{Diff}_\vol$-invariant metrics possesses a remarkable geometric property: there is no loss or gain of regularity during the geodesic evolution. 
Our result is a generalization of the classical no-loss-no-gain result for fully right invariant metrics, first proved by Ebin and Marsden~\cite{EbMa1970}.

\begin{lemma}[no-loss-no-gain]\label{lem:nolossnogain}
Let $F\colon T\Diff(M)\to TT\Diff(M)$ be a $\operatorname{Diff}_\vol$-equivariant vector field on $T\Diff(M)$.
Assume that $F$ extends to a smooth vector field $F^s$ on $T\mathcal D^{s}(M)$
for all $s \geq s_0$ for some $s_0 > \frac{d}{2}+1$ and let
$J_s$
denote the maximal interval of existence of the solution to the corresponding 
initial value problem on $T\mathcal D^{s}(M))$ with initial conditions $(\varphi_0,v_0)\in T\mathcal D^{s}(M)$.

If $(\varphi_0,v_0)\in T\mathcal D^{s+1}(M)$ then $J_{s+1}(\varphi_0,v_0)=J_s(\varphi_0,v_0)$, i.e., the flow of the vector field $F$ has no loss or gain in regularity on its maximal interval of existence.
\end{lemma}

The proof of this result follows the lines of the proof for geodesic sprays of right invariant metrics in \cite[Thm.~12.1 and Lem.~12.2]{EbMa1970} with only minor adaptations. 
In fact \cite[Lem.~12.2]{EbMa1970} is already formulated for all
of $\mathcal D^{s}(M)$ and uses only the invariance with respect to divergence free vector fields. 
\begin{proof}
We start by proving the following claim, which is essentially \cite[Lem. 12.2]{EbMa1970}.

\smallskip
\noindent\textbf{Claim A} (\cite[Lem. 12.2]{EbMa1970}):  Let $\varphi \in \mathcal D^{s}(M)$.  If
$T\varphi.X: M\to TM$ is an $H^s$-map for all $X\in T_{\operatorname{id}}  \mathcal D^{s}_{\vol}(M)$ then $\varphi\in \mathcal D^{s+1}(M)$. 
\smallskip

Let $p\in M$. We choose an open neighborhood $U$ of $p$ and coordinates $x^i$ such that $\vol |_{U}=dx^1\wedge\ldots\wedge dx^d$. 
The idea of the proof is to construct  divergence free vector fields $X$,  that are locally acting as the $i$-th derivative, i.e., $T\varphi.X=\frac{\partial}{\partial_{x_i}}\varphi$.
Therefore let $\lambda$ be a smooth function with support in $U$, that is constant one on a smaller neighborhood $V\subset U$.
Consider the vector field $X$ on $U$ via
\begin{equation}
X = \left( x_2\lambda_{x_2}+\lambda, - x_2\lambda_{x_1},0,\ldots,0   \right)\;,
\end{equation}
where $\lambda_{x_i}$ denotes the $i$-th partial derivative of $\lambda$. Since $X$ has compact support in $U$, it can be extended to a smooth vector field to all of $M$ by letting it zero outside of $U$. 
A direct calculation shows that $\operatorname{div} X=0$ and thus $X\in  T_{\operatorname{id}}  \mathcal D^{s}_{\vol}(M)$. On the smaller neighborhood $V$ the vector field $X$ is $(1,0,\ldots,0)$ and thus
\begin{equation}\label{eq:tvarphiX}
T\varphi.X|_V=\frac{\partial}{\partial_{x_1}}\varphi\;.
\end{equation}
Here we used that  $T\varphi.X|_V=\sum_{i=1}^d \frac{\partial}{\partial_{x_i}}(\varphi) X^i$ on $U$. 
Using \eqref{eq:tvarphiX} it follows that $T\varphi.X$ is $H^s(V)$ if and only if $\frac{\partial}{\partial_{x_i}}\varphi$ is of class $H^s$. Now the statement follows by iterating the argument for all other coordinates.

Let $\operatorname{Flow}^s_t(\varphi_0,v_0)$ be the vector flow of the smooth vector field $F^s$ on  $T\mathcal D^{s}(M))$
with  initial conditions $(\varphi_0,v_0)\in T\mathcal D^{s}(M)$. 
Recall that $\operatorname{Flow}^s_t(\varphi_0,v_0)$ is defined for $t\in J_s(\varphi_0,v_0)$. 
We need to prove that for initial conditions $(\varphi_0,v_0)\in T\mathcal D^{s+1}(M)$
the flow $\operatorname{Flow}^s_t(\varphi_0,v_0) \in \mathcal D^{s+1}(M)$ for all $t\in J_s(\varphi_0,v_0)$.
The strategy is to prove that $T\varphi.X$ for $\varphi=\operatorname{Flow}^s_t(\varphi_0,v_0)$ is an $H^s$-map for each $X\in\Xcal_\mu(M)$ and then use Claim~A.

Let $\eta(\tau)$ be the one parameter subgroup generated by some divergence free vector field $X\in T_{\operatorname{id}} \mathcal D^s_{\vol}(M)$. 
Using the invariance of the vector field (flow resp.) under volume preserving diffeomorphisms we have 
\begin{equation}
\underbrace{\operatorname{Flow}^s_t(\varphi_0,v_0)}_{\varphi}\circ\eta(\tau)=\operatorname{Flow}^s_t(\varphi_0\circ\eta(\tau),v_0\circ\eta(\tau))
\end{equation}
Differentiating this equation at $\tau=0$ we obtain 
\begin{equation}
	T\varphi.X = \frac{d}{d\tau}\bigg|_{\tau=0}\operatorname{Flow}^s_t(\varphi_0\circ\eta(\tau),v_0\circ\eta(\tau)).
\end{equation}
Now, the right-hand side is an $H^s$-map since $T\operatorname{Flow}^s_t\colon TT\mathcal D^s(M)\to T\mathcal D^s(M)$ is smooth and 
\begin{equation}
	\frac{d}{d\tau}\bigg|_{\tau=0}(\varphi_0\circ\eta(r),v_0\circ\eta(r)) \in TT\mathcal D^s(M)
\end{equation}
since $(\varphi_0,v_0)\in T\mathcal D^{s+1}(M)$.
Thus, the left-hand side $T\varphi.X$ also has to be an $H^s$-map and the result follows from Claim~A.
We have now shown that $J_{s}(\varphi_0,v_0)\subset J_{s+1}(\varphi_0,v_0)$. By definition $J_{s+1}(\varphi_0,v_0)\subset J_s(\varphi_0,v_0)$ and thus the result follows. 
\end{proof}

\section{Smoothness lemma}\label{app:smooth_lemma}
The following lemma is fundamental in the proof of smoothness of the metric and spray on the Sobolev completion $\mathcal D^s(M)$.
Let, as before, $T^{q}\mathcal D^s(M)$ denote the vector bundle above $\mathcal D^s(M)$ whose fibres are tangent vectors in the $H^q$ Sobolev class.

\begin{lemma}\label{lem:smoothrho}
Let $F\colon P^\infty(M)\times \Xcal(M)\to \Xcal(M)$ be a smooth, possibly nonlinear, differential operator of order $l-1$ in its first argument and $l$ in its second argument ($l\geq 1$).

If $s>\frac{d}{2}+l$ then $F$ extends to a smooth operator $P^{s-1}(M)\times \Xcal^s(M)\to \Xcal^{s-l}(M)$ and the mapping
\begin{equation}
	T\mathcal D^s(M)\to T^{s-l}\mathcal D^s(M);\quad (\varphi,h)\mapsto F(\det(D\varphi^{-1}),h\circ\varphi^{-1})\circ\varphi %
\end{equation}
is smooth.
\end{lemma}

\begin{proof}
	The definition of $\rho$ is
	\begin{equation}
		\rho = \det(D\varphi^{-1}).
	\end{equation}
	Calculating the derivative of the identity $x=(\varphi^{-1}\circ\varphi)(x)$ yields
	\begin{align}
		\operatorname{1}=D(\varphi^{-1}\circ\varphi)=D(\varphi^{-1})\circ\varphi\cdot D\varphi
	\end{align}
	Composing by $\varphi^{-1}$ and taking the determinant and we get
	\begin{equation}
		\rho = \frac{1}{\det(D\varphi)\circ\varphi^{-1}}
	\end{equation}
	so
	\begin{equation}
		\rho\circ\varphi = \frac{1}{\det(D\varphi)}.
	\end{equation}
	Smoothness of the 
	mapping $r\colon\varphi\mapsto \rho\circ\varphi$ follows directly from the Banach algebra property of $H^{s-1}(M)$ and the positivity 
	of $\det(D\varphi)$.

	Denoting by $\partial_i$ partial differentiation with respect to a choice of local coordinates, we want to show that $\varphi\mapsto (\partial_i \rho)\circ\varphi$ is smooth as a mapping $\mathcal D^s(M)\to H^{s-2}(M)$.
	For this, we write the mapping as
	\begin{equation}
		\varphi \mapsto (\partial_i (r(\varphi)\circ\varphi^{-1}))\circ\varphi
	\end{equation}
	where $r\colon \mathcal{D}^s \to H^{s-1}(M)$ is smooth.
	Differentiating through we get
	\begin{equation}\label{eq:product}
		(\partial_i r(\varphi)) \Big((\partial_i \varphi^{-1})\circ\varphi\Big).
	\end{equation}
	Clearly, $\varphi\mapsto\partial_i r(\varphi)$ is smooth as a mapping $\mathcal D^s(M) \to H^{s-2}(M)$.
	That $(\partial_i \varphi^{-1})\circ\varphi$ is smooth follows since $D(\varphi^{-1})\circ\varphi = (D\varphi)^{-1}$ and since inversion of an invertible matrix is a smooth operation due to the Banach algebra property of~$H^s$.
	It also follows that the product \eqref{eq:product} is smooth.
	By iterating this process with $r\colon \varphi\to (\partial^{k-1}_{i_1,\ldots,i_{k-1}}\rho) \circ\varphi$ we get that $\varphi\mapsto (\partial^k_{i_1\cdots i_k}\rho)\circ\varphi$ is smooth as a mapping $\mathcal D^s(M)\to H^{s-k-1}(M)$.

	That $(\varphi,h)\mapsto (\partial_{i_1,\ldots,i_k} (h\circ\varphi^{-1}))\circ\varphi$ is smooth as a mapping $T\mathcal D^s(M)\to T^{s-k}\mathcal D^s(M)$ is well known (see \cite[App.~2]{EbMa1970}).

	Notice that so far we have only used the Banach algebra property of $H^{s-1}$, but not of $H^{s-2}$, $H^{s-3}$, etc.

	Finally, the differential operator $F$ can locally be written
	\begin{equation}
		F(\rho,u)(x) = f(\rho(x),\partial_{1}\rho,\ldots,\partial^{l-1}_{d\ldots d}\rho, u(x),\partial_{1}u,\ldots,\partial^l_{d\ldots d}u)
	\end{equation}
	for some finite-dimensional smooth mapping $f\colon \RR_{>0}\times\RR^{m-1}\times (TM)^m \to TM$ of all the partial derivatives.
	Thus, locally, we have
	\begin{equation}
		(\varphi,h)\mapsto f(\rho\circ\varphi,(\partial_{1}\rho)\circ\varphi,\ldots,(\partial^{l-1}_{d\ldots d}\rho)\circ\varphi, u(x),(\partial_{1}(h\circ\varphi^{-1}))\circ\varphi,\ldots,(\partial^l_{d\ldots d}(h\circ\varphi^{-1}))\circ\varphi).
	\end{equation}
	Since each term plugged into $f$ is (at least) in $H^{s-l}$, and $s-l>d/2$, it follows from the $\omega$-Lemma (see, e.g., \cite[Sec.~2]{EbMa1970}) that $F$ is smooth.
\end{proof}

\begin{remark}
	The condition $s>d/2+l$ cannot be weakened. 
	Take, for example, 
	\begin{equation}
		F(\rho,u) = \abs{\nabla^{l-1}\rho}^2 u.
	\end{equation}
	If $\rho\in P^{s-1}(M)$ then $\nabla^{l-1}\rho$ belongs to $H^{s-l}$.
	Although $\rho\mapsto \nabla^l\rho$ is smooth as a mapping $H^{s-1}\to H^{s-l}$, unless $s-l>d/2$ the product $\nabla^{l-1}\rho \mapsto \abs{\nabla^{l}\rho}^2$ is not smooth as a mapping $H^{s-l}\to H^{s-l}$.
\end{remark}

\def\cprime{$'$}

\end{document}